\documentclass[11pt,letterpaper]{amsart}
\usepackage{amssymb,amsfonts,amsthm,array,epsfig,graphics,graphicx,tabularx}

\usepackage[usenames,dvipsnames]{color}

 \def\RR{{\mathbb R}}  \def\TT{{\mathbb T}}
   
 \def\ZZ{{\mathbb Z}}

\def\Wi{\widetilde}

   \def\cO{{\mathcal O}}

\def\cF{{\mathcal F}}

\newtheorem{theorem}{{Theorem}}[section]
\newtheorem{proposition}[theorem]{{Proposition}}

\newtheorem{lemma}[theorem]{{Lemma}}

\newtheorem{question}[theorem]{{Question}}

\theoremstyle{definition}

\theoremstyle{remark}
\newtheorem{remark}[theorem]{{Remark}}

\title[Self orbit equivalences on a class of Anosov flows]{Self orbit equivalences on a class of Anosov flows}
\author{Bin Yu}
\date{\today}

\begin{document}
\maketitle

\begin{abstract}
Let $M_k$ ($k\in \ZZ$ and $|k|>4$) be the $3$-manifold obtained by doing $k$ Dehn surgery on the figure-eight knot, and $X_t$ be the canonical Anosov flow on $M_k$
that is constructed by Goodman. The main result of this article is that
if $|k|\gg 0$, then $\mbox{Mod} (M_k) \cong \ZZ_2\oplus \ZZ_2$ and every element of the mapping class group of $M_k$ can be represented by a self orbit equivalence  of  $X_t$.
\end{abstract}

\section{Introduction}\label{s.Int}
Let $Y_t$ be the suspension Anosov flow induced by the vector field
$(0,\frac{\partial}{\partial t})$ on the sol-manifold $W:=\TT^2 \times [0,1]/(x,1)\sim (A(x),0)$, where $A=\left(\begin{array}{cc}
                                         2 & 1 \\
                                         1 & 1 \\
                                       \end{array}
                                     \right)$ is an Anosov automorphism on $\TT^2$.
Define an orientation on $W$.
The first class of 
Anosov flows on hyperbolic 3-manifolds was constructed by Goodman \cite{Go}
by doing
a type of dynamical Dehn surgery, namely Dehn-Fried-Goodman surgery (abbreviated as DFG surgery), along a periodic orbit of $Y_t$. For the detail about DFG surgery, we refer to Goodman \cite{Go}, Fried \cite{Fr} and Shannon \cite{Sha}. 
Now let us explain Goodman’s examples more clearly. Let $\omega$ be the periodic
orbit of $Y_t$ associated to the origin $O$ that is  the fixed point of $A$. 
Then after doing a $k$-DFG ($k\in \ZZ, |k|>4$)\footnote{Following Thurston \cite{Thu}, the condition $|k|>4$ is to ensure that $M_k$ is a hyperbolic $3$-manifold.} surgery on $Y_t$ at $\omega$,
we get a new Anosov flow $X_t$ on the new three manifold $M_k$. Here $M_k$ is the hyperbolic manifold that can be obtained from $W$ by doing $k$-DFG surgery at $\omega$. 

The Anosov flow $X_t$ on $M_k$ shares some impressive properties, for instance:
\begin{enumerate}
\item by Fenley \cite{Fen}, $X_t$ is skew $\RR$-covered, which is an interesting
topological property of the orbit space;
\item by Fenley \cite{Fen} and Barthelme-Fenley \cite{BF0}, each periodic orbit of 
$X_t$ is isotopic
to infinitely many periodic orbits of $X_t$;
\item by \cite{Yu} due to the author of this paper, up to orbital equivalence,
$X_t$ is the unique Anosov flow on $M_k$.
\end{enumerate}

This paper is devoted to understand the self orbit equivalences of $X_t$, in particular
their relations with the mapping class group of $M_k$. 
To understand the self orbit equivalences of Anosov flows is not only an important topic in the study of Anosov flows (\cite{BG}, \cite{BM}), but also plays an important role in the study of partially hyperbolic diffeomorphisms in 3-manifolds,
see for instance, \cite{BFP}, \cite{FP1}, \cite{FP2}.

The main result of this paper is the following:

\begin{theorem}\label{t.main}
If $|k|\gg 0$, then $\mbox{Mod} (M_k) \cong \ZZ_2\oplus\ZZ_2$ and every element of the mapping class group of $M_k$ can be represented by a self orbit equivalence  of the Anosov flow $X_t$ on $M_k$.\footnote{Let $M$ be an orientable manifold, in this paper, we use $\mbox{Mod} (M)$ to represent the mapping class group of $M$ and $\mbox{Mod}^+(M)$ to represent the subgroup of $\mbox{Mod} (M)$ that consists of the orientation preserving elements of $\mbox{Mod}(M)$.}
\end{theorem}

This theorem gives a positive answer to the following question asked by Barthelme and Mann (\cite{BM}).

\begin{question}\label{q.BM}[Question $2$ of \cite{BM}]
Does there exist an ($\RR$-covered or not) Anosov flow on a hyperbolic 3-manifold
$M$ such that every element of the mapping class group of $M$ is represented by a self orbit equivalence?
\end{question}

Certainly, on this topic, it is still interesting to look for an Anosov flow on a hyperbolic $3$-manifold 
such that there exists an element of the mapping class group of the $3$-manifold
that can not be represented by any self orbit equivalence.

Now we sketch the proof of Theorem \ref{t.main}.
Firstly observe that  when $|k|\gg 0$, every homeomorphism $f$ on $M_k$ must be orientation preserving (Lemma \ref{l.orpr}),
and up to isotopy, $f(\omega)=\omega$ (Lemma \ref{l.fix}). Then up to isotopy, we can suppose $f(N)=N$ and $f(V)=V$ where $V$ is a solid torus neighborhood of
$\omega$ and $N$ is the closure of $M_k\setminus V$ that is homeomorphic to
the figure-eight knot exterior. It is a classical result that $\mbox{Mod} (N)\cong D_4$
where $D_4$ is the dihedral group.
We choose four linear automorphism $g_0= id_W$, $g_1$, $g_2$ and $g_3$ on 
$W$ such that they generate a $\ZZ_2 \oplus \ZZ_2$ subgroup of $\mbox{Mod}^+ (W)$ (Lemma \ref{l.gZ4}). Here each of $g_0$ and $g_2$ is a self orbit equivalence of $Y_t$, and each of $g_1$ and $g_3$ is an orbit equivalence between $Y_t$ and $Y_{-t}$. By using an idea of Fried, $g_i$ can induce
a `good' homeomorphism $h_i$ on the figure-eight knot exterior $N$. Then each $h_i$ induces two
isotopic homeomorphisms $f_i$ and $f_i'$ on $M_k$.
$f_i'$ is defined by a natural extension of $h_i$ on $M_k = N\cup V$. 
$f_i$ is induced by $h_i$ by collapsing $\partial N$ to $\omega$ in $M_k$. Notice that
in this procedure, $(W, \omega, Y_t)$ is transformed to $(M_k, \omega, X_t)$ under $k$-DFG surgery at $\omega$. Due to the constructions of  $g_i$ and $h_i$, we have that
$f_i$ ($i=0,2$)  is a self orbit equivalence of $X_t$, and $f_i$ (i=1,3) is an orbit equivalence between $X_t$ and $X_{-t}$. Now the left part of the proof of Theorem \ref{t.main} consists of the following three points.
\begin{enumerate}
\item Every homeomorphism $f$ on $M_k$ is isotopic to some $f_i$ (Lemma \ref{l.1}). This is based on the
two observations in the first part of this paragraph and some combinatorial topology discussions. 
\item For every $i=0, 1,2,3$, there exists a self orbit equivalence $F_i$
of $X_t$  such that $F_i$ is isotopic to $f_i'$ (Lemma \ref{l.2}). Here $F_i = f_i$ if $i=0,2$ and 
$F_i=\eta \circ f_i$ if $i=1,3$, where $\eta$ is a special orbit equivalence between
$X_t$ and $X_{-t}$ that is isotopic to $id$ on $M_k$. $\eta$ is defined by Fenley
\cite{Fen}
and Barbot \cite{Ba} in their orbit space theory for three dimensional Anosov flows, see Section \ref{s.Skew}.
\item Every $f_i$ (i=1,2,3) is not isotopic to $id$ on $M_k$. The point here is to prove that $f_2$ is not isotopic to $id$ on $M_k$ (Lemma \ref{l.3}), where we essentially use 
a result (Theorem \ref{t.BG}) due to Barthelme and Gogolev \cite{BG}. 
\end{enumerate}

\section*{Acknowledgments}
We would like to thank Yi Shi and Youlin Li, In particular, we thank Youlin Li for his very valuable suggestions about the proof of Lemma \ref{l.ale}.
The author is supported by Shanghai Pilot Program for Basic Research, National Program for Support of Top-notch Young Professionals and the Fundamental Research Funds for the Central Universities.

\section{Skew $\RR$-covered Anosov flows}\label{s.Skew}
We assume the reader know the basic facts about $3$-dimensional
Anosov flows. More related information we refer to \cite{Fen} and \cite{Bart}.
In this section, we will briefly recall some facts of skew $\RR$-covered Anosov flows
that will be used in the proof. 
 Most of the materials mentioned in this section basically are due to Barbot \cite{Ba} and Fenley \cite{Fen}. The brief introduction here mainly  borrows from Barthelme and Mann \cite{BM}. More details  we refer to   the nice survey due to Barthelme \cite{Bart}.

Let $\phi_t$ be an Anosov flow on a closed orientable $3$-manifold $M$ with coorientable stable foliation $\cF^s$ and $\cF^u$, and let $\Wi M$ be the universal cover of $M$. $\phi_t$  is called $\RR$-covered
if the leaf space of  its stable foliation\footnote{I.e. the leaf space of the lifting foliation of the stable foliation on $\Wi M$.} is homeomorphic to $\RR$.
\footnote{By Barbot \cite{Ba} or Fenley \cite{Fen}, it is equivalent that the leaf space of the unstable foliation is homeomorphic to $\RR$.} By  Barbot \cite{Ba} and Fenley \cite{Fen}, an $\RR$-covered Anosov flow on a closed orientable $3$-manifold is
either orbitally equivalent to the suspension of an Anosov automorphism on $\TT^2$ or \emph{skew}.   Here we say that $\phi_t$ is a skew $\RR$-covered Anosov flow
if the orbit space of the lift of the flow to $\Wi M$ is homeomorphic to
the infinite diagonal strip
$$\cO = \{(x,y)\in \RR^2 \mid |x-y|<1\}$$
via a homeomorphism taking the stable leaves of the flows to the horizontal cross sections of the strip, and unstable leaves to the vertical cross sections. See Figure \ref{f.skewR}.  

\begin{figure}[htp]
\begin{center}
  \includegraphics[totalheight=6.6cm]{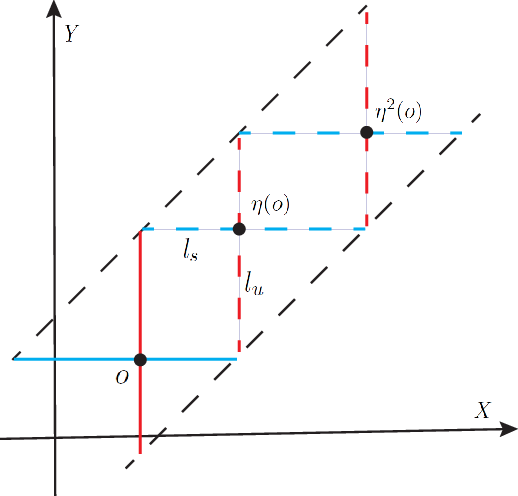}\\
  \caption{The skew $\RR$-covered orbit space $\cO$ and the map $\eta$}\label{f.skewR}
\end{center}
\end{figure}

There is a very special  continuous fixed point free map $\eta:\cO \to \cO$,
which is called the \emph{half-step up} map by Barthelme and Mann \cite{BM}. $\eta$
is defined as follows.
Take a point $o\in \cO$, let $l_s$ be a stable leaf that is the upper boundary of
the strip consisting of the unstable leaves through $o$, and $l_u$ be a unstable leaf that is the upper boundary of
the strip consisting of the stable leaves through $o$. We define by $\eta (o)$ the intersection point of $l_s$ and $l_u$. This map exchanges stable leaves and unstable leaves, so $\tau=\eta^2$ preserves the stable leaves (resp. unstable leaves).
$\eta$ can induce an orbit equivalence between $\phi_t$ and $\phi_{-t}$ that is homotopic to $id$ on $M$. For simplicity, we  still call the orbit equivalence by $\eta$.  Then $\eta^2$ is a self-orbit equivalence of $\phi_t$ that is homotopic to $id$ on $M$. 
\begin{remark}\label{r.etaid}
If $M$ is a closed hyperbolic $3$-manifold, then each of $\eta$ and $\eta^2$ is isotopic to $id$ on $M$. This is a direct consequence of a classical result due to Gabai, Meyerhoff and N. Thurston \cite{GMT}, which says that when $M$ and $N$ are two closed hyperbolic $3$-manifolds and $f,g:M\to N$ are two homotopic homeomorphisms, then $f$ is isotopic to $g$. 
\end{remark}

Moreover, $\eta$ is important to describe the homotopy class of the periodic
orbits of $\phi_t$:

\begin{proposition}[Barbot \cite{Ba}, Fenley \cite{Fen}]\label{p.eta}
Let $\phi_t$ be a skew $\RR$-covered Anosov flow on a closed orientable $3$-manifold,
and $\alpha$ be a periodic orbit of $\phi_t$. Then all the orbits freely homotopic to a 
periodic orbit $\alpha$  
 are given by $\{\eta^{2k} (\alpha), k\in \ZZ\}$,
and 
all the orbits freely homotopic to the inverse of $\alpha$ are given by  
$\{\eta^{2k+1} (\alpha), k\in \ZZ\}$.
Moreover, if $M$ is a hyperbolic $3$-manifold, every two periodic orbits 
$\eta^{2k_1} (\alpha)$ and $\eta^{2k_2}(\alpha)$ ($k_1\neq k_2$) are different, and therefore there are infinitely many periodic orbits that are homotopic to $\alpha$.
\end{proposition} 

Notice that in this case, we can replace free homotopy by isotopy by the following theorem due to Barthelme and Fenley \cite{BF0}.
\begin{theorem}\label{t.hopiso}
Let $\phi_t$ be a skew $\RR$-covered Anosov flow on a closed orientable $3$-manifold.
If two  periodic orbits of $\phi_t$  are  freely  homotopic,  then  in  fact  they are isotopic.
\end{theorem}

In the proof of Theorem \ref{t.main},  we also need the following result due to 
Barthelme and Gogolev \cite{BG}. 
\begin{theorem}\label{t.BG}
Let $\phi_t$ be a skew $\RR$-covered Anosov flow on a closed orientable $3$-manifold $M$, and $h$ be a self orbit equivalence of $\phi_t$ that is isotopic to $id$ on $M$.
Then there exist an integer $k$ such that $h\circ \eta^{-2k}$ preserves every orbit of
$\phi_t$. Moreover, if $M$ is a hyperbolic $3$-manifold, $k$ above is unique. 
\end{theorem}

\section{A representation of $\mbox{Mod}^+ (N)$}\label{s.ModN}
Firstly, we define a homeomorphism $g_1: W\to W$ by $g_1((x,y),t) =((x_1,y_1),1-t)$
where $\left(\begin{array}{ccc}
                                         x_1 \\
                                         y_1 
                                       \end{array}
                                     \right)=B_1\left(\begin{array}{ccc}
                                         x \\
                                         y 
                                       \end{array}
                                     \right)$
and $B_1=  \left(\begin{array}{cc}
                                         -1 & 0 \\
                                         1 & 1 \\
                                       \end{array}
                                     \right)$. $g_1$ is well-defined that is essentially
due to the fact that $B_1 A B_1^{-1}= A^{-1}$. It is easy to check that $g_1$ is an orbit equivalence between the Anosov flows $Y_t$ and $Y_{-t}$ on $W$.
Define $g_2: W\to W$ by $g_2((x,y),t) =((x_2,y_2),t)$
where $\left(\begin{array}{ccc}
                                         x_2 \\
                                         y_2 
                                       \end{array}
                                     \right)=B_2\left(\begin{array}{ccc}
                                         x \\
                                         y 
                                       \end{array}
                                     \right)$ and $B_2 =  \left(\begin{array}{cc}
                                         -1 & 0 \\
                                         0 & -1 \\
                                       \end{array}
                                     \right)$.
It is also easy to check that $g_2$ is a self orbit equivalence of $Y_t$.
Similarly define  $g_3: W\to W$ by $g_3((x,y),t) =((x_3,y_3),1-t)$
where $\left(\begin{array}{ccc}
                                         x_3 \\
                                         y_3 
                                       \end{array}
                                     \right)=B_3\left(\begin{array}{ccc}
                                         x \\
                                         y 
                                       \end{array}
                                     \right)$
and $B_3=  \left(\begin{array}{cc}
                                         1 & 0 \\
                                         -1 & -1 \\
                                       \end{array}
                                     \right)$.
Similar to the case for $g_1$, $g_3$ is also an orbit equivalence between the Anosov flows $Y_t$ and $Y_{-t}$. Moreover define by $g_0$  the identity map on $W$.

\begin{lemma}\label{l.gZ4}
$\mathcal{G}:=\{[g_0], [g_1], [g_2],[g_3]\}$ is a subgroup of $\mbox{Mod}^+ (W)$ that is isomorphic to $\ZZ_2\oplus \ZZ_2$.  Here $[g_i]$ ($i=0,1,2,3$) represents the isotopy class of $g_i$.
\end{lemma}
\begin{proof}
Based on some elementary calculations, we can get that $B_i^2 =\left(\begin{array}{cc}
                                         1 & 0 \\
                                         0 & 1 \\
                                       \end{array}
                                     \right)$ ($i=1,2,3$) and $B_iB_j=B_jB_i=B_k$
where  $(i,j,k)$ is any permutation of $(1,2,3)$. These equalities ensure that 
$g_i^2 =Id_W =g_0$ ($i=1,2,3$) and $g_ig_j =g_jg_i =g_k$ where $(i,j,k)$ is any permutation of $(1,2,3)$.  By the above equalities, we can build a 
homomorphism $\varphi: \ZZ_2 \oplus \ZZ_2 \to \mathcal{G}$ by $\varphi((0,0))=[g_0]$, 
$\varphi((1,0))=[g_1]$, $\varphi((1,1))=[g_2]$ and $\varphi((0,1))=[g_3]$.
Then to prove the lemma, it only needs to prove that each of $g_1$, $g_2$ and $g_3$ is not isotopic to $g_0=id$. 
Notice that  $\omega$ can represent a free element of $H_1(W)$
since $\omega$ intersects to a fiber torus once. Then $g_1(\omega)=-\omega$ and $\omega$
are not isotopic in $W$, therefore $g_1$ is not isotopic to $g_0=id$.  For the same reason, one can get that $g_3$ is not isotopic to $g_0=id$.
$\{(\frac{1}{4},0), (\frac{1}{2},\frac{1}{4}), (\frac{1}{4}, \frac{3}{4})\}$
and $\{(\frac{3}{4},0), (\frac{1}{2},\frac{3}{4}), (\frac{3}{4}, \frac{1}{4})\}$ are two periodic orbits of $A$ which
correspond to two different periodic orbits $\beta_1$ and $\beta_2$ of $Y_t$.
Since  $\left(\begin{array}{ccc}
                                         \frac{3}{4} \\
                                         0 
                                       \end{array}
                                     \right)=B_2\left(\begin{array}{ccc}
                                         \frac{1}{4}  \\
                                         0 
                                       \end{array}
                                     \right)$ on $\TT^2$, 
therefore, $g_2(\beta_1)=\beta_2$. By Nielsen fixed point theory (see, e.g. \cite{Ji}), $\beta_1$ and $\beta_2$ are not freely homotopic, therefore $g_2$
is not isotopic to $g_0=id$. 
\end{proof}

Blow up $O\in \TT^2$ to a circle $S_o^1$ such that 
a point $x\in S_o^1$ corresponds to a vector $\vec{v}\in T_o^1 (\TT^2)$.
And $\TT^2$ is blowed up to a punctured torus $T_0$ such that,
\begin{enumerate}
\item $\partial T_0 = S_o^1$;
\item every point $x\in S_o^1$ is the starting point of the ray  associted to the ray $l_x$ in $\TT^2$ that starts at $O$ and is parallel to $\vec{v}$, where $\vec{v}\in T_o^1 (\TT^2)$  
and corresponds to $x$.
\end{enumerate}
Naturally we can identify $\TT^2 \setminus \{O\} $ with $T_0 \setminus S_o^1$
by a homeomorphism $\phi: \TT^2 \setminus \{O\}\to T_0 \setminus S_o^1$.
The self homeomorphism $\phi A \phi^{-1}$ on $T_0 \setminus S_o^1 $
 can be extended to a homeomorphism $\overline{A}: T_0\to T_0$ as follows.
Recall that every point $x\in S_o^1$ corresponds to a ray  $l_x$ in $\TT^2$ 
that starts at $O$. $A$ maps $l_x$ to another ray $l$ starting at $O$.
Let $y$ be the end of $\phi (l\setminus \{O\} )$. Then $\overline{A}(x):=y$.
One can automatically check that  $\overline{A}$ is a homeomorphism on $T_0$.
 For every $i\in\{1,2,3\}$, the self homeomorphism $\phi B_i \phi^{-1}$ on $T_0 \setminus S_o^1 $ can be similarly extended to a homeomorphism $\overline{B_i}: T_0\to T_0$.

We can identify $N$ with $\mbox{Map} (T_0, \overline{A})$ where $\mbox{Map} (T_0, \overline{A})$ is the standard mapping torus of $\overline{A}$ on $T_0$. 
Take an oriented circle $l\subset \partial N$ that is the boundary of some $T_0\times \{t_0\}$ in $N$.
Take another oriented circle $m \subset \partial N$ that intersects each circle $m_t =\partial T_0\times \{t\}$ ($t\in[0,1]$) at one point $q_t$ that is associated to a stable separatry of
$A$ at $O$. 

Notice that $B_1AB_1^{-1}=A^{-1}$ and the fact that $\overline{B}_1$ and $\overline{A}$
are the continuous extension of $\phi B_1 \phi^{-1}$ and $\phi A \phi^{-1}$ respectively, then we have that $\overline{B}_1\overline{A}\overline{B}_1^{-1}=\overline{A}^{-1}$.
Define a self homeomorphism
$h_1$ on $N$ by $h_1 (p, t):= (\overline{B}_1(p), 1-t)$. $h_1$ is well-defined since
$\overline{B}_1\overline{A}\overline{B}_1^{-1}=\overline{A}^{-1}$.
Similarly, one can define a homeomorphism $h_3 (p, t):= (\overline{B_3}(p), 1-t)$ on $N$
and a homeomorphism $h_2 (p, t):= (\overline{B_2}(p), t)$ on $N$. We define $h_0=id\mid_N$. 

\begin{proposition}\label{p.hZ4}
$\mbox{Mod}^+ (N)=\mathcal{H}$  that is isomorphic to
$\ZZ_2\oplus \ZZ_2$.  Here $\mathcal{H}:= \{[h_0], [h_1], [h_2],[h_3]\}$ and $[h_i]$ ($i=0,1,2,3$) represents the isotopy class of $h_i$.
\end{proposition}
\begin{proof}
Define   $g_4: W\to W$ by $g_4((x,y),t) =((x_4,y_4),1-t)$
where $\left(\begin{array}{ccc}
                                         x_4 \\
                                         y_4 
                                       \end{array}
                                     \right)=B_4\left(\begin{array}{ccc}
                                         x \\
                                         y 
                                       \end{array}
                                     \right)$
and $B_4=  \left(\begin{array}{cc}
                                         0 & 1 \\
                                         -1 & 0 \\
                                       \end{array}
                                     \right)$. 
$g_4$ is well-defined since $B_4 A B_4^{-1}=A^{-1}$.
It is easy to check that $g_4$ is an orientation-reversing homeomorphism on $W$. Similar the construction of 
$h_i$ ($i=1,2,3$), $g_4$ can induce an orientation-reversing homeomorphism $h_4$
on $N$. This means that $\mbox{Mod} (N)$ contains an orientation-reversing element
$[h_4]$.  Notice that the mapping class group of the figure-eight knot exterior  $\mbox{Mod}(N)\cong D_4$ (see for instance  Table 14.2 of \cite{Mar}) where $D_4$ is the dihedral group  that is the symmetry group of the square. Then $|\mbox{Mod}(N)|=8$. Since $\mbox{Mod} (N)$ contains an orientation-reversing element, then $|\mbox{Mod}^+(N)|=4$. By Lemma \ref{l.gZ4} and its proof,
we  know that $\mathcal{G}=\{[g_0], [g_1], [g_2],[g_3]\}\cong \ZZ_2 \oplus \ZZ_2$, and $g_i^2 =Id_W =g_0$ ($i=1,2,3$) and $g_ig_j =g_jg_i =g_k$ where $(i,j,k)$ is any permutation of $(1,2,3)$. Moreover $h_i$ on $N$
is a continuous extension of $\Phi g_i \Phi^{-1}$ on $N\setminus \partial N$.
Then we have that $h_i^2 =h_0=Id\mid_N$ ($i=1,2,3$) and $h_ih_j =h_jh_i =h_k$ where $(i,j,k)$ is any permutation of $(1,2,3)$. Due to these equalities, similar to the proof of Lemma \ref{l.gZ4},  we can construct a surgective homomorphism from $\ZZ_2\oplus \ZZ_2 $ to $\mathcal{H}$. Further recall that 
$|\mbox{Mod}^+(N)|=4$, to complete the proof of the proposition,
 we only need to show that each of $h_1$, $h_2$ and $h_3$ is not isotopic to
$h_0= id_N$. This essentially is the same to the proof that $g_i$ ($i=1,2,3$)
is not isotopic to $g_0=id$ on $W$. 
Below we provide the details of a proof.

Due to its definition, $h_1$ maps $m$ to $-m$. Further notice that $m$ intersects  the once-punctured torus  $T_0\times \{t_0\}$ once, then by using intersection number,
one can easily prove that $h_1$ is not isotopic to
$h_0= id_N$. Similarly we have that $h_3$ is not isotopic to  $h_0= id_N$.
In the proof of Lemma \ref{l.gZ4}, we know that there are two disjoint oriented knots 
$\beta_1$ and $\beta_2$ in $W$ such that, 
\begin{enumerate}
\item each of them is disjoint with $\omega$;
\item $g_2 (\beta_1)=\beta_2$
\item $\beta_1$ and $\beta_2$ are not freely homotopic in $W$.
\end{enumerate}
Let $\beta_1'=\Phi(\beta_1)$ and $\beta_2'=\Phi(\beta_2)$ that are two disjoint
oriented knots in $N$. It is easy to check that $\Phi:(W\setminus \omega,\beta_1, \beta_2)=(N\setminus \partial N,\beta_1', \beta_2')$.  Notice that $\beta_1$ and $\beta_2$ are not freely homotopic in $W$, so $\beta_1$ and $\beta_2$ are not freely homotopic in $W\setminus \omega$ and therefore $\beta_1'$ and $\beta_2'$ are not freely homotopic in
$N\setminus \partial N$. Hence  $\beta_1'$ and $\beta_2'$ are not freely homotopic in $N$. Further observe that $h_2 (\beta_1')=\beta_2'$ since $g_2 (\beta_1)=\beta_2$ and 
$h_2= \phi g_2 \phi^{-1}$ in $N\setminus \partial N$.
Then $\beta_1'$ and $h_2 (\beta_1')$ are not freely homotopic in $N$, therefore
$h_2$ is not isotopic to
$h_0= id_N$. The proposition is proved.
\end{proof}

\section{The homeomorphisms $f_i$ and $f_i'$ on $M_k$}\label{s.home}
$\Phi: W\setminus \omega \to N\setminus \partial N$ maps the flow $Y_t$ on $W\setminus \omega$ to a flow  $\Phi(Y_t)$ on $N\setminus \partial N$. Then we can get the Anosov flow $X_t$ on $M_k$ from $Y_t$ through a DFG surgery. For our purpose, here we just introduce this procedure in Fried's way (\cite{Fr}). Due to Fried \cite{Fr}, $\Phi(Y_t)$ can be continuously extended to a continuous flow $Z_t$ on $N$ such that
$Z_t\mid_{\partial N}$ is a nonsingular Morse-Smale flow with $4$ periodic orbits
with alternating orientations. (see Figure \ref{f.Frisg}) Take a circle bundle $E$ on $\partial N$ such that each circle fiber $e$ is transverse to the flowlines of $Z_t\mid_{\partial N}$
and is isotopic to $l+km$ in $\partial N$. Here $m$ is a meridian circle of the figure-eight knot exterior $N$ and $l$ is a longitute circle of $N$ that is homology vanishing. We remark that up to isotopy, each of $m$ and $l$ is unique. By pinching every circle fiber
$e$ to a point $x_e$, we blow down $N$ to the manifold $M_k$ by the pinching map $\pi: N\to M_k$. Moreover,  
the flow $Z_t$ on $N$ is blowed down to the Anosov flow $X_t$ on $M_k$. See Figure \ref{f.Frisg}.

\begin{figure}[htp]
\begin{center}
  \includegraphics[totalheight=6cm]{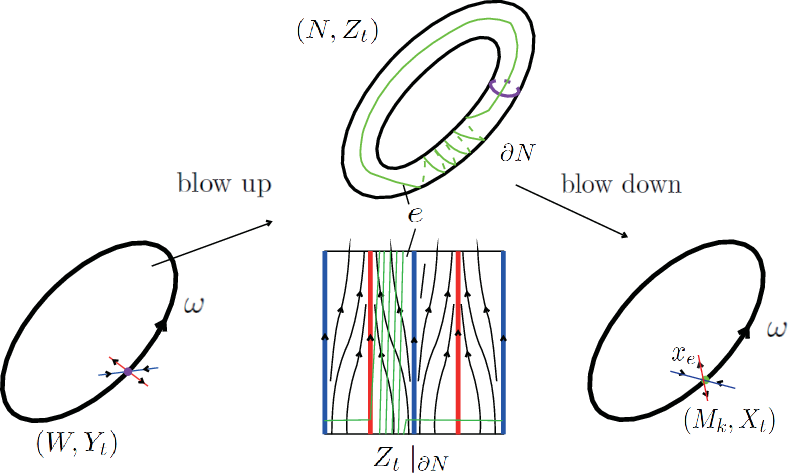}\\
  \caption{Fried surgery in the case $k=5$}\label{f.Frisg}
\end{center}
\end{figure}

Since $g_2$ is a self orbit equivalence of $Y_t$, then $\Phi g_2 \Phi^{-1}$ is a self
orbit equivalence of $\Phi(Y_t)$ on $W\setminus \omega$. Moreover, since $Z_t$ is a continuous extension of $\Phi(Y_t)$, then  as a continuous extension of $\Phi g_2 \Phi^{-1}$ on $N$, $h_2$ is is a self
orbit equivalence of $Z_t$ on $N$. Further notice that up to isotopy, $h_2:m\to m$ and $h_2:l\to l$. Then up to isotopy, $h_2$ preserves the circle fibration structure $E$ on
$\partial N$. Then, after composing a self orbit equivalence preserving every orbit if necessary, we can assume that  $h_2$ preserves $E$.
Then $h_2$ naturally induces a self orbit equivalence $f_2$ of $X_t$ on $M_k$
such that $\pi  h_2 = f_2 \pi$.
Notice that $g_i$ ($i=1,3$) is an  orbit equivalence between $Y_t$ and $Y_{-t}$, and $h_i:m\to -m$
and $h_i:l\to -l$. Then similarly, after composing a self orbit equivalence preserving every orbit if necessary,  $h_i$ naturally induces an  orbit equivalence $f_i$ between  $X_t$ and $X_{-t}$ on $M_k$
such that $\pi  h_i = f_i \pi$. $h_0=id\mid_N$ naturally induces  $f_0=id$ on $M_k$.
We can  conclude the above discussion to the following lemma.

\begin{lemma}\label{l.fi}
For every $i\in \{0,1,2,3\}$, after composing a self orbit equivalence preserving every orbit of $Z_t$ if necessary,  the self homeomorphism $h_i$ on $N$ induces a self homeomorphism $f_i$ on $M_k$ such that $\pi  h_i = f_i \pi$. Moreover,
\begin{enumerate}
\item if $i=0$ or $i=2$, $f_i$ is a self orbit equivalence of $X_t$ on $M_k$; 
\item if $i=1$ or $i=3$, $f_i$ is an orbit equivalence between  $X_t$ and $X_{-t}$ on $M_k$.
\end{enumerate}
\end{lemma}

Let $V$ be a solid torus tubular neighborhood of $\omega$, then the closure of $M_k\setminus V$ is homeomorphic to $N$. For simplicity, we still call it $N$.
Then we can think that $M_k=V\cup_{\psi_k} N$ where $\psi_k: \partial V \to \partial N$
is the corresponding gluing map such that up to isotopy, $\psi_k (m_v)=l+km$. Here 
$m_v$ is a meridian circle of the solid torus $V$. Now we extend the homeomorphism 
$h_i$ ($i=0,1,2,3$) on $N$ to $f_i'$ on $M_k$. First, $h_0=id\mid_N$, it can be naturally
extended to $f_0'=f_0=id$ on $M_k$.
In the other three cases, since up to isotopy $h_i$ maps $m$ to $m$ or $-m$,
there is no obstruction to extend $h_i$ to a homeomorphism $f_i'$ on $M_k$ such that
$f_i'$ maps $\omega$ to itself.
Moreover, it is easy to observe that $f_i'$ and $f_i$ are isotopic.

\section{Proof of the main theorem}\label{s.Pro}
The following lemma was firstly used in the study of Anosov flows by  Bowden and Mann (\cite{BM}). Both of the statement and the proof  essentially follow  \cite{BM}.
\begin{lemma}\label{l.fix}
If $|k|\gg 0$, then for every self homeomorphism $f$ on $M_k$, up to isotopy,
$f(\omega)=\omega$. 
\end{lemma}
\begin{proof}
Due to \cite{Thu}, up to isotopy, $\omega$ is the unique shortest closed geodesic in
the hyperbolic three manifold $M_k$. Due to Mostow rigidity, up to isotopy, $f$ is an isometry on $M_k$, and therefore $f(\omega)$ is a shortest closed geodeic in $M_k$.
But  $\omega$ is the unique shortest closed geodesic in $M_k$, hence up to isotopy,
$f(\omega)=\omega$.
\end{proof}

\begin{lemma}\label{l.orpr}
If $|k|\gg 0$, we endow an (in fact, any) orientation on $M_k$. Then for every self homeomorphism $f$ on $M_k$, up to isotopy, $f$ preserves the orientation on $M_k$. 
\end{lemma}
\begin{proof}
Due to Lemma \ref{l.fix}, up to isotopy, we can assume that $f(\omega)=\omega$.
Recall that $M_k = V\cup_{\psi_k} N$. Here $V$ is a solid torus with a core $\omega$, $N$ is the figure-eight knot exterior (see Section \ref{s.ModN}) and
$\psi_k: \partial V \to  \partial N$ is the corresponding gluing homeomorphism.
Since $f(\omega)=\omega$, up to isotopy, we can assume $f(V)=V$ and $f(N)=N$. Further recall that
up to isotopy, $\psi_k (m_v)= l + k m$ where $m_v$ is an oriented merdian circle of $V$ and 
$m$ and $l$ are the standard oriented meridian circle and  oriented longitute circle in $\partial N$.

We assume by contradiction that $f$ is not orientation preserving, then 
$f\mid_N$ is also not orientation preserving. Since up to isotopy, as the meridian of the figure-eight knot exterior and as the unique longitute which is homology vanishing, each of $m$ and $l$ is unique. This means that in this case, up to isotopy, either $f(l)=-l$ and $f(m)=m$, or $f(l)=l$ and $f(m)=-m$. Further since up to isotopy, $\psi_k (m_v)= l + k m$ on $\partial N$, then $f(\psi_k (m_v))$ is isotopic to $\pm (l-km)$ on $\partial N$. This is impossible since $f(\psi_k (m_v))$ must bounds a disk in $V$
in $M_k= V\cup_{\psi_k} N$, but such a circle must isotopes to either $l+km$ or $-l-km$
in $\partial N$. The proof of the lemma is complete.
\end{proof}

\begin{lemma}\label{l.1}
Let $f$ be  a self homeomorphism on $M_k$, then $f$ is isotopic to one of $f_0=id$, $f_1$, $f_2$ and $f_3$. 
\end{lemma}
\begin{proof}
By Lemma \ref{l.fix} and  Lemma \ref{l.orpr}, $f$ is oriention-preserving and up to isotopy, we can assume that $f(V,\omega) =(V,\omega)$ and $f(N)=N$.
Then by Proposition \ref{p.hZ4}, $f\mid_N$ is isotopic to some $h_i=f_i'\mid_N$.
Define $g:= (f_i')^{-1}\circ f$ that is a homeomorphism on $M_k$ such that
$g(V,\omega) =(V,\omega)$, $g(N)=N$ and $g\mid_N$ is isotopic to $f_0\mid_N=id_N$.
Since $f_i$ and $f_i'$ are isotopic (see Section \ref{s.home}), to complete the proof of this lemma, we are left to show that $g$ is isotopic to $f_0=id$ on $M_k$. From now on, we focus on showing
this.

Since $g:N \to N$ is isotopic to $f_0\mid_N=id_N$, then up to isotopy, we can assume $g(m)=m$ and $g(l)=l$. This fact permits us to assume that $g\mid_V =id_V$ up to isotopy. Let $V_1$ be a compact tubular neighborhood of $\omega$ in the interior of $V$.
Let $U$ be the cloure of $V-V_1$ that can be paramerized by $\TT^2 \times [0,1]$ such that $\TT^2 \times \{0\}$ and $\TT^2 \times \{1\}$ correspond to $\partial N =\partial V$ and $\partial V_1$ respectively. 

Let $F_t: N\to N$ ($t\in[0,1]$) be an isotopy such that $F_0 =id_N$ and $F_1 =g\mid_N$.
$F_t$ can be extended to an isotopy $F_t'$ on $M_k$ such that
\begin{enumerate}
\item
$F_t'=F_t$ on $N$, 
\item $F_t'(x,s)=(F_{t(1-s)}(x),s)$ on $U$ where $x\in \TT^2$ and $s\in [0,1]$,
\item $F_t' =id$ on $V_1$.
\end{enumerate}
One can automatically check that $F_t'$ is an isotopy between
$F_0' =f_0$ and $F_1'$ on $M_k$. Furthermore, it is easy to check that $F_1'$ and 
$g$ are isotopic on $M_k$. Therefore, $g$ is isotopic to $f_0=id$ on $M_k$. The proof 
of the lemma is complete.
\end{proof}

\begin{lemma}\label{l.2}
$f_0$ and $f_2$ are two self orbit equivalences of $X_t$, and  $F_1:=\eta\circ f_1$ and $F_3:=\eta\circ f_3$
are two self orbit equivalences of $X_t$. Moreover, $F_1$ is isotopic to $f_1$ and $F_3$ is isotopic to $f_3$. 
\end{lemma}
\begin{proof}
By Lemma \ref{l.fi}, each of $f_0$ and $f_2$ is a self orbit equivalence of $X_t$,
 and  each of $f_1$ and $f_3$ is a self orbit equivalence between $X_t$ and $X_{-t}$.
By Section \ref{s.Skew}, $\eta$ is also an orbit equivalence between $X_t$ and $X_{-t}$.
Therefore, each of $F_1=\eta\circ f_1$ and $F_3=\eta\circ f_3$ is a self orbit equivalence of $X_t$. Morever, notice that $\eta$ is isotopic to $f_0=id$ on $M_k$ (see Remark \ref{r.etaid}), therefore $F_1$ is isotopic to $f_1$ and $F_3$ is isotopic to $f_3$. 
\end{proof}

\begin{lemma}\label{l.3}
Every $f_i$ ($i=1,2,3$) is not isotopic to $f_0= id$. 
\end{lemma}
\begin{proof}
Firstly we prove the lemma in the case $i=2$.
Assume by contradiction that $f_2$ is isotopic to $f_0= id$, then by Theorem \ref{t.BG}, there is a 
unique $k\in \ZZ$ such that $h=f_2\circ \eta^{-2k}$ is a self orbit equivalence of $X_t$
such that $h$ preserves each orbit of $X_t$. Moreover, by Proposition \ref{p.eta}, for every $i\in \ZZ\setminus \{0\}$,
$h\circ \eta^{2i}$ does not preserve each periodic orbit of $X_t$. Notice that
$f_2 (\omega)=\omega$, therefore $k=0$ and $h=f_2$. This means that $f_2$ must preserve
each periodic orbit of $X_t$. 

Similar to the proof of Lemma \ref{l.gZ4}, $\{(\frac{1}{4},0), (\frac{1}{2},\frac{1}{4}), (\frac{1}{4}, \frac{3}{4})\}$
and $\{(\frac{3}{4},0),$ $ (\frac{1}{2},\frac{3}{4}), (\frac{3}{4}, \frac{1}{4})\}$ are two periodic orbits of $A$ which
correspond to two different periodic orbits $\beta_1$ and $\beta_2$ of $X_t$.
Since  $\left(\begin{array}{ccc}
                                         \frac{3}{4} \\
                                         0 
                                       \end{array}
                                     \right)=B_2\left(\begin{array}{ccc}
                                         \frac{1}{4}  \\
                                         0 
                                       \end{array}
                                     \right)$ on $\TT^2$, 
 therefore $f_2 (\beta_1)=\beta_2$.
This conflicts to the fact that $f_2$ must preserve
each periodic orbit of $X_t$.  Then, the assumption is not correct and therefore
$f_2$ is not isotopic to $f_0= id$. 

Now we prove the lemma in the case $i=1$.
By Lemma \ref{l.fi},  $f_1$  is a self orbit equivalence between $X_t$ and $X_{-t}$.
By the construction, $f_1(\omega) =\omega^{-1}$. Assume that $f_1$ is isotopic to $f_0=id$, then the self orbit equivalence $F_1=\eta f_1$ of $X_t$ is also isotopic to
$f_0=id$. Then by Theorem \ref{t.BG}, there is a 
unique $k\in \ZZ$ such that $h= \eta^{2k}F_1= \eta^{2k+1}f_1$ preserves each periodic orbit of $X_t$. In particular $\eta^{2k+1}f_1 (\omega)=\omega$. Recall that $f_1(\omega) =\omega^{-1}$, then  $\eta^{2k+1}(\omega)=\omega^{-1}$, and therefore 
$\eta^{2(2k+1)}(\omega)=\omega$. Since $2k+1\neq 0$, due to Proposition \ref{p.eta}, $\eta^{2(2k+1)}$ does not  preserve each periodic orbit of $X_t$. Then we get a contradiction and 
therefore
$f_1$ is not isotopic to $f_0= id$. 
One can prove the lemma in the case $i=3$ in the same way to the case that $i=1$.
The proof of the lemma is complete.
\end{proof}

The result of the following lemma should be classical, but we have not found
a related reference, so we give a proof here.

\begin{lemma}\label{l.ale}
Let $f$ be a homeomorphism on the solid torus $V$ such that $f\mid_{\partial V}=id\mid_{\partial V}$, then $f$ is isotopic to $id$ on $V$ relative to $\partial V$.
\end{lemma}
\begin{proof}
Let $m$ be a meridian circle in $\partial V$ that bounds a meridian disk
$D$ in $V$. By an isotopy relative to $\partial V$ if necessary, we can
assume that the interior of $f(D)$ intersects with $D$ at finitely many pairwise simple closed curves. Notice that $V$ is an irreducible $3$-manifold, then by a standard `inner most disk' combinatorial topology trick, by a further isotopy relative to $\partial V$  if necessary, we can assume that the interior of $f(D)$ is disjoint  with $D$. Therefore $S=f(D)\cup D$ is a two sphere in $V$ such that $S\cap V =m$ and $S$ bounds a three ball in $V$.  Then after an isotopy relative to $\partial V$, we can assume that $f(D)=D$.

Since $f\mid_{\partial D} = id_{\partial D}$, then by using two dimentional Alexander trick,
$f\mid_{\partial D}$  is isotopic to $id_{\partial D}$  relative to $\partial D$ by an isototy $h_t$ on $D$ such that $h_0=id_D$ and $h_1 =(f\mid_D)^{-1}$. Find
a small compact neighborhood $U(D)$ of $D$ such that $U(D)\cap \partial V =m$ and
$U(D)$ can be parameterized by $D\times [-1,1]$ modular $(x,s_1)=(x,s_2)$ ($s_1,s_2 \in [-1,1]$). See Figure \ref{f.DUDV} as an illustration.  Define an isotopy $H_t$ on $V$ by $H_t (y)=y$ if $y\notin U(D)$ and
$H_t (x,s)= (h_{(1-|s|)t}(x),s)$  ($x\in D$ and $s\in [-1,1]$).
Define $G_t:=H_t \circ f$. Then one can automatically check that $G_0=f$ and
$(G_1)\mid_{\partial V \cup D} =id_{\partial V \cup D}$, and $G_0=f$ and $G_1$
are isotopic relative to $\partial V$. 
Then by using three dimentional Alexander trick on the three ball which is the path closure of $V\setminus (\partial V \cup D)$, we have that $G_1$ is isotopic to $id_V$  relative to $\partial V$. Therefore, $f=G_0$ is isotopic to $id$ on $V$ relative to $\partial V$.
\end{proof}

\begin{figure}[htp]
\begin{center}
  \includegraphics[totalheight=4.8cm]{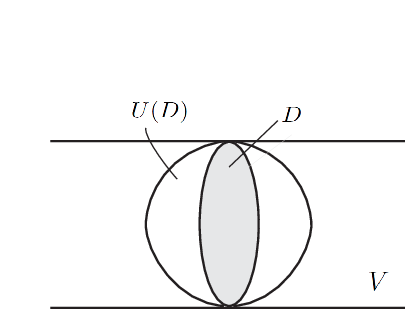}\\
  \caption{Some notations in the proof of Lemma \ref{l.ale}}\label{f.DUDV}
\end{center}
\end{figure}

The conclusion in the following lemma is elementary but useful.

\begin{lemma}\label{1.extiso}
Let $\phi_0$ and $\phi_1$ be two homeomorphisms on $M_k$
such that
\begin{enumerate}
\item $\phi_i (N)=N$ and $\phi_i (V)=V$ for $i=1,2$;
\item  $(\phi_0)\mid_N$ and $(\phi_1)\mid_N$ are isotopic, and 
$(\phi_0)\mid_V$ and $(\phi_1)\mid_V$ are isotopic.
\end{enumerate} 
Then $\phi_0$ and $\phi_1$ are isotopic on $M_k$.
\end{lemma}
\begin{proof}
Let $F_t$ ($t\in[0,1]$) be an isotopy on $N$ such that $F_0=(\phi_0)\mid_N$ and 
$F_1=(\phi_1)\mid_N$.
Let $G_t$ ($t\in[0,1]$) be an isotopy on $V$ such that $G_0=(\phi_0)\mid_V$ and 
$G_1=(\phi_1)\mid_V$.
Obviously, for $i=0$ or $i=1$, $(\phi_i)\mid_{\partial N}= \psi_k (\phi_i)\mid_{\partial V} (\psi_k)^{-1}$.
Define $h_t:=(\psi_k)^{-1} F_t \psi_k (G_t)^{-1}: \partial V \to \partial V$.
It can be automatically checked that $h_t$ is an isotopy on $\partial V$ such that
$h_0=h_1= id\mid_{\partial V}$.

Now we paramterize a small closed collar neighborhood $U(\partial V)$ of $\partial V$ in $V$ by $\partial V \times [0,1]$ such that $\partial V \times \{1\}$ identifies to $\partial V$. Let $V_0$ be the closure of $V\setminus U(\partial V)$. Define an isotopy $H_t$ on $V$ by  $H_t=id_{V_0}$ on $V_0$ and $H_t(s,x)= (s,h_{st}(x))$ ($s,t\in[0,1]$) on $U(\partial V)$. It is easy to check that $H_0=id$ on $V$ and $(H_1)\mid_{\partial V}=id\mid_{\partial V}$. 

Then $H_t G_t$ is an isotopy on $V$ such that 
$(H_t G_t)\mid_{\partial V}= h_t G_t = (\psi_k)^{-1} F_t \psi_k$
and $H_0 G_0= G_0=(\phi_0)\mid_V$ and $(H_1 G_1)_{\partial V} =(G_1)_V=(\phi_1)_V$. Then we can define an isotopy $K_t$ on $M_k$ by $K_t=F_t$ on $N$ and $K_t=H_t G_t$ on $V$
between $\phi_0$ and $\phi_1'$, where $\phi_1' = \phi_1$ on $N$ and $\phi_1'=H_1 G_1$
on $V$. 

Define a homeomorphism $g:= \phi_1^{-1} \phi_1'$ on $M_k$. Then $g\mid_N =id_N$, and
by Lemma \ref{l.ale},  it is easy to get that $g$ is isotopic to $id$ on $M_k$.
Equivalently, $\phi_1'$ is isotopic to $\phi_1$. 
Thereroe, $\phi_0$ and $\phi_1$ are isotopic on $M_k$.
\end{proof}

\begin{lemma}\label{l.4}
$[f_i]^2 =[f_0]=[id]$ ($i=1,2,3$) and $[f_i] [f_j] =[f_j] [f_i] =[f_k]$ where $(i,j,k)$ is any permutation of $(1,2,3)$ and $f_0 =id$. Here $[f_i]$ represents the isotopy class of $f_i$. 
\end{lemma}
\begin{proof}
By the final part of Section \ref{s.home}, we know that $f_i$ ($i=0,1,2,3$) is isotopic
to $f_i'$ on $M_k$. Therefore we only need to prove the lemma by replacing $f_i$ with $f_i'$. Since $f_i'$ is an extension of $h_i$ on $N$, due to the proof of Proposition  \ref{p.hZ4}, we know that 
$(f_i')^2\mid_N =f_0'\mid_N=Id\mid_N$ ($i=1,2,3$) and $f_i'f_j' =f_j'f_i' =f_k'$ where $(i,j,k)$ is any permutation of $(1,2,3)$. Now we can use  Lemma \ref{1.extiso} to complete the proof. For simplicity, we will only show it for the case that $(f_i')^2$ is isotopic to $f_0'=id$, and the left cases can be similarly proved. 
Since $(f_i')^2\mid_N =Id\mid_N$,  one can automatically check that
$(f_i')^2\mid_V$ is isotopic to $f_0'\mid_V=id_V$. Then due to  Lemma \ref{1.extiso},
$(f_i')^2$ and $f_0' =id$ are isotopic on $M_k$.
\end{proof}

Now we are prepared enough to prove the main theorem.
\begin{proof}[Proof of Theorem \ref{t.main}]
By Lemma \ref{l.4}, similar to the proof of Lemma \ref{l.gZ4}, we can build a surjective
homomorphism $\varphi: \ZZ_2 \oplus \ZZ_2 \to \cF$ by $\varphi((0,0))=[f_0]$, 
$\varphi((1,0))=[f_1]$, $\varphi((1,1))=[f_2]$ and $\varphi((0,1))=[f_3]$ where $\cF=\{[f_0], [f_1], [f_2], [f_3]\}$. Further by Lemma \ref{l.3}, we know that $\varphi$
is also injective. Therefore $\varphi: \ZZ_2 \oplus \ZZ_2 \to \cF$ is an isomorphism.
Furthermore, by Lemma \ref{l.1}, each self homeomorphism $f$ on $M_k$ is 
isotopic to some $f_i$ ($i=0,1,2,3$). Then  following the above discussions, one can easily get that $\mbox{Mod}(M_k)=\cF\cong \ZZ_2\oplus \ZZ_2$.

By Lemma \ref{l.2}, in each isotopy class of  $f_i$ ($i=0,1,2,3$), there exists a self orbit equivalence of $X_t$ ($f_i$ if $i=0,2$ and  $F_i=\eta f_i$ if $i=1,3$) in this class. This means that every element in $\mbox{Mod}(M_k)$ can be represented by a self orbit equivalence of $X_t$.  The proof of Theorem \ref{t.main} is complete.
\end{proof}

\vskip 1cm

\noindent Bin Yu

\noindent {\small School of Mathematical Sciences}

\noindent {\small Key Laboratory of Intelligent Computing and
Applications (Tongji University), Ministry of Education}

\noindent{\small Tongji University, Shanghai 200092, CHINA}

\noindent{\footnotesize{E-mail: binyu1980@gmail.com }}

\vskip 2mm

\end{document}